\documentclass[12pt]{amsart}

\usepackage{amsmath,amssymb,latexsym,amsthm,newlfont,enumerate}

\usepackage{graphicx}
\usepackage[all]{xy}
\makeindex

\theoremstyle{plain}

\newtheorem{thm}{Theorem}[section]

\newtheorem{pro}[thm]{Proposition}

\newtheorem{lem}[thm]{Lemma}

\newtheorem{cor}[thm]{Corollary}
\newtheorem{con}[thm]{Conjecture}

\theoremstyle{definition}

\newtheorem{dfn}[thm]{Definition}

\newtheorem{nt}[thm]{Notation}

\newtheorem{rem}[thm]{Remark}

\newtheorem{exa}[thm]{Example}

\theoremstyle{remark}

\newcommand{\Z}{\mathbb{Z}}
\newcommand{\N}{\mathbb{N}}
\newcommand{\C}{\mathbb{C}}

\newcommand{\Q}{\mathbb{Q}}
\newcommand{\PS}{\mathbb{P}}
\newcommand{\F}{\mathbb{F}}

\DeclareMathOperator{\Pic}{Pic} 
\DeclareMathOperator{\rk}{rank} \DeclareMathOperator{\W}{Weil}

\DeclareMathOperator{\Exc}{Exc} \DeclareMathOperator{\Spec}{Spec}
 
 \DeclareMathOperator{\Bl}{Bl}
\DeclareMathOperator{\Bs}{Bs}

\DeclareMathOperator{\Sing}{Sing}

\DeclareMathOperator{\Proj}{Proj}

\DeclareMathOperator{\NE}{NE}
\DeclareMathOperator{\NEb}{\overline{\mathrm{NE}}}

\DeclareMathOperator{\Cl}{Cl}

\begin{document}

\bibliographystyle{alpha}

\title{The defect of Fano $3$-folds}
\date{}
\author{Anne-Sophie Kaloghiros}
\address{Department of Pure Mathematics and Mathematical Statistics, 
Uni\-ver\-si\-ty of Cambridge, Wilberforce Road, Cambridge CB3 0WB, Uni\-ted
Kingdom}
\email{A.S.Kaloghiros@dpmms.cam.ac.uk}
\setcounter{tocdepth}{1}
\maketitle
\begin{abstract}
This paper studies the rank of the divisor class group of terminal Gorenstein Fano $3$-folds. If $Y$ is not $\Q$-factorial, there is a small modification of $Y$ with a second extremal ray; Cutkosky, following Mori, gave an explicit geometric description of the contractions of extremal rays on terminal Gorenstein $3$-folds.
I introduce the category of weak-star Fanos, which allows one to run the Minimal Model Program (MMP) in the category of Gorenstein weak Fano $3$-folds.
If $Y$ does not contain a plane, the rank of its divisor class group can be bounded by running a MMP on a weak-star Fano small modification of $Y$. 
These methods yield more precise bounds on the rank of $\Cl Y$ depending on the Weil divisors lying on $Y$.
I then study in detail quartic 3-folds that contain a plane and give a general bound on the rank of the divisor class group of quartic $3$-folds.
Finally, I indicate how to bound the rank of the divisor class group of higher genus terminal Gorenstein Fano 3-folds with Picard rank 1 that contain a plane.
\end{abstract}

\tableofcontents
\section{Introduction}
\label{sec:introduction}

Let $Y_4 \subset \PS^4$ be a quartic hypersurface with terminal singularities.
The Grothendieck--Lefschetz theorem states that Cartier divisors on $Y$
are restrictions of Cartier divisors on $\PS^4$, i.e.~ that $\Pic Y \simeq \Z[ \mathcal{O}_Y(1)]$. However, no such result holds for $\Cl Y$, the
group of Weil divisors of $Y$. If the quartic $Y$ is not
factorial, very little is understood about some aspects of its topology.
\par 
More generally, a Fano $3$-fold $Y$ with terminal Gorenstein
singularities is a $1$-parameter flat
deformation of a nonsingular Fano $3$-fold with the same Picard
rank \cite{Nam97a}: these are classified in 
\cite{Isk77,Isk78,MM86}. However, as in the case of terminal quartic $3$-folds,
the rank of $\Cl Y$ is not known in general.
\par 
A normal projective variety is \emph{$\Q$-factorial} if an integral multiple of
every Weil divisor is Cartier. A terminal Gorenstein Fano
$3$-fold $Y$ is $\Q$-factorial if and only if it is factorial \cite[Lemma 6.3]{Kaw88}, that
is, if and only if
\[
H^2(Y,\Z) \simeq H_4(Y,\Z).
\]    
Factoriality is a global topological property: it depends both on the
analytic type of singularities and on their relative position. 
If $Y_4 \subset \PS^4$ is a nodal quartic hypersurface with at most
$8$ singular points, then $Y$ is factorial \cite{Ch06}. 
Whereas the presence of a small number of ordinary double points does not
affect the factoriality of a hypersurface, this is
not true for even
slightly more complicated types of singularities. For instance, a
quartic $Y_4$ in 
the linear system $\Sigma$ spanned by
the monomials \[\{x_0^4, x_1^4, (x_4^2x_3+x_2^3)x_0,
x_3^3x_1, x_4^2x_1^2\}\] on $\PS^4$ is
not factorial because the plane $\Pi{=}\{x_0{=}x_1{=}0\}$ lies on $Y$. Yet, the general element $Y_4 \in \Sigma$ has a unique  c$A_1$ singular
point $P=(0{:}0{:}0{:}0{:}1)$ \cite{Me04}. 
\par 
In this paper, I study the rank of the divisor class group, or equivalently the \emph{defect},  of terminal Gorenstein Fano
$3$-folds with Picard rank $1$. The defect of a terminal Gorenstein Fano $3$-fold $Y$ is defined as the
difference:
\[
\sigma(Y)= \rk \Cl Y- \rk \Pic Y= h_4(Y)-h^2(Y).
\] 
The defect is a global topological invariant that measures how far $Y$
is from being factorial, or, in other words, to what extent Poincar\'e
duality fails on $Y$. 
\par  
The notion of defect was introduced
in \cite{C83} in an attempt, based on Deligne's Mixed Hodge Theory
\cite{D74}, to extend Griffiths' results \cite{G69} to the Hodge
theory of double covers of $\PS^3$ branched in a nodal hypersurface. 
Works on the defect of hypersurfaces or of Fano
$3$-folds have focussed on
studying their mixed Hodge structures \cite{C83,Di90,NS95, Cy}. 
These results rely on computations of the cohomology of specific
varieties: they are impractical to determine a global bound on the
defect of terminal Gorenstein Fano $3$-folds. In this paper, I use
birational geometry to bound the
defect of terminal Gorenstein Fano $3$-folds. These methods also yield more precise bounds on the defect in terms of the Weil divisors lying on the $3$-fold. 

\par 
A nodal quartic hypersurface in $\PS^4$ has at most $45$ ordinary
double points \cite{V83}; the Burkhardt quartic (Example~\ref{exa:1})
is the unique such quartic up to projective
equivalence\cite{dJSBVV90}. One sees that 
the defect of the Burkhardt quartic is at least $15$. I show that the
defect of any quartic $Y_4 \subset \PS^4$ with terminal singularities
is at most $15$
and that this bound can be refined in some cases. More precisely, I prove:  
\begin{thm}
 \label{thm:2}
Let $Y_4 \subset \PS^4$ be a quartic hypersurface with terminal
singularities. The rank of $\Cl Y$ satisfies:
\begin{itemize}
\item[(i)]  $\rk \Cl Y \leq 9$ when $Y$ contains neither a plane nor a quadric,
\item[(ii)] $\rk \Cl Y \leq 10$ when $Y$ does not contain a plane,
\item[(iii)] $\rk \Cl Y \leq 16$, with equality precisely when $Y$ is projectively equivalent to the Burkhardt quartic. 
\end{itemize}
\end{thm}

\subsection*{Outline}
I now shetch the proof of Theorem~\ref{thm:2} and present the organisation of the paper.  

In Section~\ref{sec:categ-weak-fano}, I define the category of
\emph{weak-star Fano} $3$ folds.
A small factorialization of a
terminal Gorenstein Fano $3$-fold $Y$ is weak-star Fano 
unless $Y$ contains a plane embedded by $\vert A_Y\vert_{\vert \PS^2}=\mathcal{O}_{\PS^2}(1)$. 
Weak-star Fano $3$-folds are inductively Gorenstein: if $\varphi \colon X \to X'$ is an extremal birational contraction, then $X'$ has terminal Gorenstein singularities. This is a major simplification: one may then use the explicit geometric description of extremal contractions of Mori--Cutkosky \cite{Cut88}.
 
In Section~\ref{sec:minim-model-progr},  I show that more is true: the
MMP may be run in the category of weak-star Fano
$3$-folds.  Since weak-star Fano $3$-folds are small modifications of terminal Gorenstein Fano $3$-folds, some of their numerical invariants--such as their anticanonical degree--can be read off the classification of nonsingular Fano $3$-folds \cite{Isk77, Isk78,MM82}. Following the description of \cite{Cut88}, it is possible to associate to each extremal contraction some numerical constraints: this reduces many questions on the MMP of weak-star Fano $3$-folds to a careful analysis of the classification \cite{Isk77, Isk78,MM82}.

Let $Y$ be a terminal Gorenstein Fano $3$-fold and $X$ a small factorialization; note that $\rk \Cl Y= \rk \Pic X=\rho(X)$. The main case is when $X$ is a weak-star Fano $3$-fold of Fano index $1$. The rank of $\Cl Y$ is bounded by running the MMP on $X$ and by studying the numerical constraints associated to divisorial contractions, especially those with centre along a curve.  Roughly, the cases of high defect for $Y$ correspond to the MMP on $X$ consisting of the contraction of several low degree surfaces and ending with a high anticanonical degree Fano $3$-fold, such as $\PS^3$ (see Sections~\ref{sec:mmp} and \ref{sec:bound-defect-some} for precise statements).
The case of Fano index $\geq 2$ is treated briefly in Section~\ref{sec:higher-index-fanos}. Since ${-}K=2H$ for $H$ Cartier, a birational extremal contraction can only be a flop or the inverse of the blowup of a smooth point.  

Section~\ref{sec:fano-3-folds-4} is independent of the rest of the paper and treats the case of terminal Gorenstein Fano $3$-folds whose small factorialization is not weak-star.  
I study in detail quartic
$3$-folds that contain a plane and bound their defect by explicit calculation. I also indicate how to
bound the defect of terminal Gorenstein Fano
$3$-folds with Picard rank $1$ and genus $g \geq 3$.  

\subsection*{A ``geometric motivation'' of non-factoriality }The framework of weak-star Fano $3$-folds allows one to state a more precise description of the divisor class group of non-factorial Fano $3$-folds \cite{Kal09}.
If $Y$ is not factorial, there is a small modification $X$ of $Y$ with a second extremal ray $R$; in most cases, $R$ is of divisorial type. Now, $X$ can be deformed to a nonsingular small modification of a Fano $3$-fold $X_{\eta}$ with $\rho(X_{\eta})=2$, on which the extremal ray $R$ is still present.  It is then possible to run a two-ray game on $X_{\eta}$, and this two-ray game deforms back to a two-ray game on $X$. This shows that $Y$ contains a surface of one of a finite number of special types (see \cite{Kal09}). This analysis can be carried out at each divisorial step of the MMP on a small factorialization of $Y$, and can be used to exhibit generators for the Cox ring of $Y$.

\subsection*{Notations and conventions}
All varieties considered in this paper are normal, projective and
defined over
$\C$. Let $Y$ be a terminal Gorenstein Fano $3$-fold, $A_Y=-K_Y$ denotes the anticanonical divisor of $Y$.  The \emph{Fano index} of $Y$ is the maximal integer such that $A_Y=i(Y) H$ with
  $H$ Cartier. As I only consider Fano $3$-folds
with terminal Gorenstein singularities in this paper, the term index
always stands for Fano index.  
  The \emph{degree} of $Y$ is $H^3$ and the \emph{genus} of $Y$ is $g(Y)=h^0(X,A_Y)-2$. I denote by $Y_{2g-2}$ for $2 \leq g \leq 10$ or $g=12$ (resp.~
$V_d$ for $1\leq d \leq 5$) terminal
  Gorenstein Fano $3$-folds of Picard rank $1$, index $1$
  (resp.~ $2$) and genus
  $g$ (resp.~ degree $d$).
Finally, $\F_m= \PS(\mathcal{O}_{\PS^1}\oplus \mathcal{O}_{\PS^1}(-m))$ denotes the
$m$th Hirzebruch surface.

\subsection*{Acknowledgements}
This paper is based on my PhD  thesis \cite{Kal07}.
I would like to thank my PhD supervisor Alessio Corti, for his
encouragement and
support. I cannot express how grateful I am for the many
discussions we have had and for his guidance over these past few years. 
I would also like to thank Miles Reid, Nick Shepherd-Barron,
Burt Totaro, Ivan Smith, Vladimir Lazi\'c and Andreas H\"oring for many useful conversations
and comments.  Lastly, I thank the referee for several helpful suggestions.

\section{The category of weak-star Fano $3$-folds}
\label{sec:categ-weak-fano}
In this section, I recall some results on terminal Gorenstein Fano $3$-folds and on the MMP for $3$-folds. I also introduce the category of \emph{weak-star Fano} $3$-folds. Most terminal Gorenstein Fano $3$-folds have a weak-star Fano small factorialization. These are inductively Gorenstein and hence well behaved under the birational operations of the MMP.

\subsection{Definitions and first results}
\label{sec:defin-first-results}
\begin{dfn}\label{dfn:1}
\mbox{}\begin{enumerate}
\item[1.]
A $3$-fold $Y$ with terminal Gorenstein singularities is \emph{Fano} if its anticanonical divisor $A_Y$ is ample.
\item[2.]
A $3$-fold $X$ with terminal Gorenstein singularities is \emph{weak Fano} if $A_X$ is
  nef and big.  
  \item[3.]The morphism $h\colon X \to Y$ defined by $\vert nA_X \vert$ for $n>\!\!>0$ is the \emph{(pluri-)anticanonical map} of $X$, $R=R(X, A)$ is the \emph{anticanonical ring} of $X$ and $Y= \Proj R$ is the \emph{anticanonical model} of $X$.
\item[4.] A weak Fano $3$-fold $X$ is a \emph{weak-star Fano} if, in addition:
  \begin{enumerate}
  \item[(i)] $A_X$ is ample outside of a finite set of curves, i.e.~  $h\colon X \to Y$ is a small modification,
  \item[(ii)]  $X$ is factorial, and in particular $X$ is Gorenstein,
  \item[(iii)] $X$ is inductively Gorenstein, that is $(A_X)^2\cdot S >1$ for every irreducible divisor $S$ on $X$,
  \item[(iv)] $\vert A_X \vert$ is basepoint free, so that $\varphi_{\vert A \vert}$ is generically finite. 
 \end{enumerate}
\end{enumerate}
\end{dfn}
\begin{rem}
\label{rweak}
\mbox{}
\begin{enumerate}
\item[1.] Item (iii) in the definition of weak-star Fano $3$-folds guarantees that if $\varphi \colon X \to X'$ is a birational extremal contraction, then $X'$ is factorial and terminal.  
Indeed, by Lemma~\ref{lem:1} and by Mori-Cutkosky's classification (see Theorem~\ref{Cut}), $X'$ is factorial and terminal unless 
$\varphi$  contracts a plane $E \simeq \PS^2$ with normal bundle $\mathcal{O}_{\PS^2}(-2)$ to a point of Gorenstein index $2$. 
In this case, by adjunction, the restriction of $A_X$ to $E$ would be ${A_X}_{\vert E}= \mathcal{O}_{\PS^2}(1)$ and $(A_X)^2 \cdot E
\leq 1$: this is impossible when $X$ is weak-star.
\item[2.] Item (iv) of the definition is a convenience, that I have included to avoid digressing throughout the paper to exclude the special cases listed in Proposition~\ref{pro:bpf}. In addition, note that when $\vert A_X\vert$ is basepoint free, either $R(X,A_X)$ is generated in degree $1$ and $A_Y$ is very ample or $\varphi_{\vert A\vert}$ is $2$-to-$1$ to $\PS^3$ or to a quadric $Q \subset \PS^4$.
\end{enumerate}
\end{rem}
\begin{lem}\label{lem:4} Let $h\colon X\to Y$ be a small factorialization of a terminal Gorenstein Fano $3$-fold $Y$ such that $\vert A_Y\vert$ is basepoint free.
\begin{enumerate}\item[1.] 
If $A_Y$ is very ample, $X$ is not weak-star Fano if and only if $Y$ contains a plane $\PS^2$ with ${A_Y}_{\vert \PS^2}=\mathcal{O}_{\PS^2}(1)$.
\item[2.] If $A_Y$ is not very ample, $X$ is not weak-star Fano if and only if $Y{=}Y_6 \subset \PS(1^4,3)$ and $Y$ contains a plane $ \{y{=}l(x_0, \cdots, x_3){=}0\}$, where $x_i$ are the coordinates of weight $1$, $y$ is the coordinate of weight
$3$ on $\PS(1^4,3)$ and $l$ is a linear form. 
\end{enumerate}
\end{lem}
\begin{proof} Assume that $A_Y$ is very ample. If $Y$ contains a plane $\Pi$ of the given form, the proper transform $S=h_{\ast}^{-1}\Pi$ of $\Pi$ satisfies $A_X^2\cdot S=1$ and $X$ is not weak-star Fano.
Conversely, if $X$ contains an irreducible divisor $S$ with $(A_X)^2\cdot S=1$, the image of $S$ is a surface in $\PS^{g+1}$, and as $A_Y=\mathcal{O}_{Y}(1)$, and its degree in $\PS^{g+1}$ is $1$.
Similarly, if $A_Y$ is not very ample, $X$ fails to be weak-star Fano precisely when $Y$ is a sextic double solid and contains a plane of the stated form. When $Y=Y(2,4)\subset \PS(1^5,4)$ is a double cover of a quadric $Q \subset \PS^4$, $X$ is weak-star unless $Y$ contains a plane of the form $\Pi{=}\{y{=}x_0{=}x_1{=}0\}$, where $x_i$ denote the coordinates of weight $1$ and $y$ is the coordinate of weight
$4$. However,  in this case, $Y$ has nonisolated singularities. 
\end{proof}
\begin{pro}
\label{pro:bpf}
Let $X$ be a weak Fano $3$-fold.  The anticanonical linear system $\vert A_X \vert$ is basepoint free except in the following cases:
\begin{enumerate}
\item[1.] $\Bs \vert A_X \vert \simeq \PS^1$ is a curve lying on the nonsingular locus of $X$. The general member $S \in \vert A_X \vert$ is a nonsingular K$3$ surface and  $\Bs \vert A_X \vert$ is a ${-}2$-curve on $S$, $X$ is called \emph{monogonal} and is one of:
\begin{enumerate}  
\item[(i)] $X= \PS^1 \times S$, where $S$ is a weighted
hypersurface of degree $6$ in $\PS(1^2,2,3)$; $X$ has Picard rank $10$,
\item[(ii)]
$X=\Bl_{\Gamma} V$, the blowup of a weighted hypersurface $V$ of degree $6$ in $\PS(1^3,2,3)$ along the plane $\Pi{=}\{ x_0{=}x_1{=}0\}$; $X$ has
Picard rank $2$.
\end{enumerate}
\item[2.] $\Bs \vert A_X \vert =\{p\}$ is an ordinary double point and $Y$ is birational to a special complete intersection $X_{2,6} \subset \PS(1^4, 2, 3)$ with a node; $Y$ is the degeneration of a double sextic. 
\end{enumerate} 
\end{pro}
\begin{proof}
This result is well known and classical; I just sketch its proof (see \cite{Kal07} for a complete proof).
Let $X$ be a weak Fano $3$-fold; Reid \cite{Re83} shows that the general section $S$ of $\vert A_X \vert $ is a K$3$ surface with no worse than rational double points. An easy extension of Saint-Donat's results \cite{SD74} on nef and big linear systems on nonsingular K$3$ surfaces applied to $\vert A_X \vert_{\vert S}$ finishes the proof. 
\end{proof}
\begin{lem}[Cone theorem for weak Fano $3$-folds]
  \label{lem:1}
Let $X$ be a weak Fano $3$-fold. The Kleiman--Mori cone $\NE(X)=
\NEb(X)$ of $X$ is a finite rational
polyhedral cone. If $R \subset \NE(X)$ is an extremal ray,
then either $A_X \cdot R >0$, or $A_X \cdot R =0$ and there is an
irreducible divisor $D$ such that $D \cdot R <0$. There is a contraction
morphism $\varphi_R$ associated to each extremal ray $R$. If $\varphi_R$ is
small, it is a flopping contraction.  
\end{lem}
  \begin{proof}
  The $3$-fold $X$ is terminal, hence by the standard cone theorem \cite[Theorem 4-2-1]{KMM87}:
\[
\NEb(X)= \NEb(X)_{A_X \leq 0}+ \sum {C_j}
\]
where the extremal rays $C_j$ are discrete in the half space $\{
A_X >0 \}$ and can be contracted. Since $A_X$ is nef, for any $z
\in N^1(X)$,  $A_X \cdot z \leq 0$ if and only if $A_X \cdot z
=0$. The anticanonical divisor $A_X$ is big: for some
integer $m>0$, $m(A_X) \sim A+D$, where $A$ is an ample divisor and
$D$ is effective. By the Nakai--Moishezon criterion for
ampleness, if $A_X \cdot z=0$ then $D \cdot z <0$. In particular, for
$0<\epsilon \ll 1 $,
\[
\NEb(X) \subset {{-}A_X}_{<0}   \cup  ({-}A_X + \epsilon D)_{<0}.
\]
By the usual compactness argument, NE(X) is a finitely generated rational
polyhedral cone. Extremal rays can be contracted by the contraction theorem
\cite[Theorem 3-2-1]{KMM87}.
Finally, if $\phi_R$ is small, $R$ flips or flops. A flipping curve $\gamma$
on a terminal $3$-fold $X$ satisfies $A_X \cdot \gamma <1$
\cite{Be85}. The anticanonical divisor is Cartier and nef: $\phi_R$ is a flopping contraction.
\end{proof}

For clarity of exposition, I recall the classification of extremal divisorial contractions established by Cutkosky following Mori.

\begin{thm}\cite[Theorems 4 and 5]{Cut88}
 \label{Cut} Let $X$ be a $3$-fold with terminal factorial singularities and $f \colon X \to X'$ the birational contraction of an extremal
ray. Assume that $f$ is not an isomorphism in codimension $1$. Then $X'$ is $\Q$-factorial and:
  \begin{enumerate}
  \item[E$1$] If $f \colon X \to X'$ contracts a surface $E$ to a
  curve $\Gamma$, $X'$ is factorial, non-singular near $\Gamma$ and $f$ is the blowup of $I_{\Gamma}$.
  The curve $\Gamma$ is l.c.i.~, has locally planar singularities and $X$ has only $cA_n$ singularities on $E$.
  \end{enumerate}
    If $f \colon X \to X'$ contracts a surface $E$ to a
  point $p$, $f$ is one of the following:
  \begin{enumerate}
\item[E$2$]  $(E, \mathcal{O}_E (E)) \simeq (\PS^2,\mathcal{O}_{\PS^2}(-1)) $, $f$ is the blowup at a non-singular point $p$.
 \item[E$3$] $(E,  \mathcal{O}_E(E)) \simeq ( \PS^1 \times \PS^1, \mathcal{O}(-1,-1))$, $f$ is the blowup at an ordinary double point $p$.
\item[E$4$]  $(E,  \mathcal{O}_E(E)) \simeq (Q,  \mathcal{O}_{Q}(-1))$, where 
  $n \geq 3$ and $Q$ is a quadric cone in $\PS^3$, $f$ is the blowup at a $cA_{n-1}$ singular point $p$.
\item[E$5$]  $(E, \mathcal{O}_E(E)) \simeq (\PS^2,\mathcal{O}_{\PS^2}(-2)) $$f$ is the blowup of a non Gorenstein point of index $2$.
\end{enumerate}
\end{thm}

\begin{thm}\cite{Nam97a}
  \label{thm:7}
Let $X$ be a small modification of a terminal Gorenstein Fano $3$-fold.  
There is a $1$-parameter flat deformation of $X$ 
\[
\xymatrix{
X\ar[r] \ar[d]& \mathcal{X} \ar[d]\\
\{0\} \ar[r]& \Delta
}
\]
such that the generic fibre $\mathcal{X}_{\eta}$ is a nonsingular small modification of a terminal Gorenstein Fano $3$-fold. The Picard ranks, the anticanonical
degrees and the indices of $X$ and $\mathcal{X}_{\eta}$ are equal.
\end{thm}

\section{MMP for weak-star Fano $3$-folds}
\label{sec:minim-model-progr}

Let $Y$ be a terminal Gorenstein Fano $3$-fold, and let $h\colon X \to Y$ be a small factorialization. 
In Section~\ref{sec:categ-weak-fano}, I showed that unless $Y$ contains a plane, $X$ is weak-star.
In this section, I show that the category of weak-star Fano $3$-folds is preserved by the birational operations of the MMP, so that the explicit geometric description of extremal contractions of Mori--Cutkosky applies.
Each type of extremal contraction imposes numerical constraints on the intersection numbers of relevant cycles in cohomology. Given that numerical invariants of terminal Gorenstein Fano $3$-folds such as their degree or genus are the same as those of non-singular Fano $3$-folds, this observation will yield a bound on the rank of the divisor class group of terminal Gorenstein Fano $3$-folds that do not contain a plane. 
\subsection{Minimal  Model Program}
\label{sec:mmp}

\begin{lem} 
\label{lem:11}Let $X$ be a small modification of a terminal Gorenstein Fano $3$-fold $Y$. Assume that $\vert A_X \vert$ is basepoint free.
Denote by $\varphi \colon X \to X'$ an extremal
divisorial contraction with centre a curve $\Gamma$ and assume that $\Exc \varphi=E$ is a Cartier divisor. 

Then $\Gamma \subset X'$  is locally a complete intersection and has planar singularities. The contraction $\varphi$ is locally the blowup of the ideal sheaf $\mathcal{I}_{\Gamma}$. In addition, the following relations hold :
\begin{align} 
(A_X)^3&=(A_{X'})^3-2(A_{X'})\cdot \Gamma-2+2p_a (\Gamma) \\
(A_X)^2 \cdot E&=
A_{X'}\cdot \Gamma+2-2 p_a(\Gamma)\\
A_{X} \cdot E^{2}&=-2+2p_a(\Gamma) \\
E^{3}&=-(A_{X'})\cdot \Gamma +2 -2 p_a(\Gamma)
 \end{align}
\end{lem}
\begin{proof} 
I only sketch the proof of this Lemma because it is similar to \cite{Cut88}. 
As the singularities of $X$ are isolated, $\varphi$ is the blow up of a nonsingular curve $\Gamma$ away from finitely points of $X'$, whose fibers contain $\Sing X$.
Consider a point $P \in \Gamma$. Let $S\in \vert A_X\vert$ be a general section. By Bertini, $S$ is nonsingular and intersects the fiber $\varphi^{-1}(P)$ in a finite set. Since $\vert A_X\vert =\vert \varphi^{\ast} A_{X'}-E \vert$,  $S'= \varphi(S)$ is an anticanonical section of $X'$ that contains $\Gamma$. The restriction $\varphi_{\vert S}\colon S \to S'$ is a finite birational map that is an isomorphism outside of $\varphi^{-1}(P)$ and in particular, $S'$ has isolated singularities. Note that $S'$ is Cartier, and since $X'$ is Cohen Macaulay, $S'$ is normal.  By Zariski's main theorem, $\varphi_{\vert S}$ is an isomorphism. This shows that $\Gamma$ has planar singularities and that $X$ is the blow up of $\mathcal{I}_{\Gamma}$.
The relations then follow from adjunction on $X'$ and from standard manipulation of the equality:
\[
A_X=\varphi^{\ast}A_{X'}-E.
\] 
 \end{proof}
\begin{thm}
\label{thm:5}
The category of weak-star Fano $3$-folds is preserved by the
birational operations of the MMP.  
\end{thm}
\begin{proof}  The category of weak-star Fano $3$-folds is clearly stable under flops. Let $X$ be a weak-star Fano $3$-fold and $\varphi
  \colon X \rightarrow X'$ a divisorial contraction. By
  Theorem~\ref{Cut}, $X'$ is $\Q$-factorial and terminal, and as is explained in Remark~\ref{rweak}, $X'$ is also Gorenstein.

\emph{Step 1.} The anticanonical divisor $A_{X'}$ is big because $\vert A_X\vert \subset \vert \varphi^{\ast}(A_{X'})\vert$.  I prove that it is also nef. The anticanonical divisors of $X$ and $X'$ satisfy:
\begin{equation}
\label{eq:1}
A_X = \varphi^{\ast}(A_{X'})-aE,
\end{equation}
where $a=2$ if $\varphi$ contracts a plane to a nonsingular point
and $a=1$ otherwise. If $\varphi$ contracts a divisor to a
point, $A_{X'}$ is nef. Assume that $\varphi$ contracts a divisor $E$ to
a curve $\Gamma$. Let $Z'$ be an irreducible
reduced curve lying on $X'$. If $Z'$ and $\Gamma$ intersect in a
$0$-dimensional set, $A_{X'} \cdot Z'\geq 0$. Thus, $A_{X'}$
can fail to be nef only if $Z'=\Gamma$ and $A_{X'} \cdot \Gamma$ is negative.  By relation $(2)$ in Lemma~\ref{lem:11}, this is impossible when $X$ is weak-star Fano because $A_{X'}\cdot \Gamma+2 \geq A_X^2\cdot E \geq 2$.
\par

\emph{Step 2.}
Note that $A_X'$ is ample outside of a finite set of curves. Indeed, by \eqref{eq:1}, if $A_{X'}\cdot C=0$ for an effective curve $C'$, then either the proper transform $C$ of $C'$ does not meet the exceptional locus and is such that $A_X\cdot C=0$ or $C'= \Gamma$ is the centre of $\varphi$. As there are finitely many such curves, $X'$ is a small modification of a terminal Gorenstein Fano $3$-fold $Y'$.  By Theorem~\ref{thm:7}, the anticanonical degree and Picard rank of $Y'$ are the same as those that appear in the classifications of \cite{Isk77, Isk78, MM82}.

As $A_X$ is basepoint free, \eqref{eq:1} shows in addition that $\Bs \vert A_{X'} \vert$ is contained in the centre of the contraction $\varphi$. 
If $\vert A_{X'}\vert$ is not basepoint free, $X'$ is one of the $3$-folds listed in Proposition~\ref{pro:bpf}. I use the following observations: $A_{X'}^3>A_X^3$ (see the proof of Lemma ~\ref{lem:4}) and if $\rho(Y')>1$, then $\rho(Y)>1$ (see the proof of Lemma~\ref{lem:3}). 
If $\Bs\vert A_{X'}\vert= \{p\}$, then $A_{X'}^3=2$, \cite{Isk77,Isk78,MM82} show that there is no $X$ with $A_X^3<2$. Similarly, it is impossible to find a weak-star $X$ with $\rho(Y)>1$ and $A_X^3$ sufficiently small for $X'$ to have $\Bs\vert A_X'\vert \simeq \PS^1$.
The linear system $\vert A_{X'}\vert$ is basepoint free.

Note in addition, that if $A_Y$ is very ample, so is $A_{Y'}$. Indeed, $R(X',A_{X'})$ is generated in degree $1$ if and only if $\Phi_{\vert A_{X'}\vert}$ is birational onto its image. Since $\Phi_{\vert A_{X}\vert}= \nu \circ \Phi_{\vert A_{X'}\vert}$ where $\nu$ is the projection from a possibly empty linear subspace, $A_{Y'}$ is very ample.

\emph{Step 3.} Let $S'$ be an irreducible divisor on $X'$ and denote by
$S$ its proper transform on $X$. By \eqref{eq:1}, 
\[
\vert A_X \vert_{\vert S}= \vert\varphi^{\ast}(A_{X'})-aE\vert_{\vert S}
\]
for some $a \in \N$. It
is naturally a subsystem of $\vert A_{X'} \vert_{S'}$, and
the inclusion is strict if $S'$ meets the centre of $\varphi$. In
particular,
\begin{equation}
(A_X)^2 \cdot S \leq (A_{X'})^2 \cdot S' 
\end{equation}
and $X'$ does not contain an irreducible divisor $S'$ with $(A_{X'})^2
\cdot S'  \leq 1$ because $X$ is a weak-star Fano $3$-fold. 
\end{proof}
The MMP can be run in the category of weak-star
Fano $3$-folds. Moreover, there are strong restrictions on the
anticanonical models of the intermediate steps.
\begin{lem}
\label{lem:3}
Let $X := X_0$ be a weak-star Fano $3$-fold whose anticanonical
model $Y_0$ has Picard rank $1$. There is a sequence of extremal contractions:
\[
\xymatrix{
X_0 \ar[r]^-{\varphi_0} \ar[d]& X_1 \ar[r]^-{\varphi_1}\ar[d] & \cdots
& X_{n-1}\ar[r]^-{\varphi_{n-1}} \ar[d]& X_n \ar[d] \\
Y_0 & Y_1 & \cdots & Y_{n-1} & Y_n
}
\]
where for each $i$, $X_i$ is a weak-star Fano $3$-fold, $Y_i$ is
its anticanonical model, and
$\varphi_i$ is a birational extremal contraction. The Picard rank of $Y_i$, $\rho(Y_i)$, is equal to $1$ for all $i$. The
final $3$-fold $X_n$ is either a Fano $3$-fold with $\rho(X_n)=1$ or an extremal Mori fibre space. 
\end{lem}
\begin{proof}
 If $\rho(X_i)>1$, there is an
  extremal ray $R_i$ that can be
  contracted. Denote by $\varphi_{i} \colon X_i \to X_{i+1}$ the
  contraction of $R_i$. If $\varphi_i$ is not birational, $X_i=X_n$ is an
  extremal Mori fibre space and there is nothing to prove.

Assume that $\varphi_i$
is birational, Theorem~\ref{thm:5} shows that $X_{i+1}$ is a
weak-star Fano $3$-fold. Let $h_i\colon X_i \to Y_i$ denote its anticanonical map.
  
  I prove that if the Picard rank of $Y_{i}$
  is equal to $1$, then the Picard rank of $Y_{i+1}$ is also $1$.
The statement is trivial when $\varphi_i$ is a small
  contraction; assume that $\varphi_i$ is divisorial and denote by $E$
  its exceptional divisor.  
  
The image $h_i(E)=\overline{E}\subset Y_i$ of $E$ by the anticanonical map 
 is a Weil non-Cartier divisor.  
 If $h_i$ was an isomorphism near $E$, the Cartier divisor
$\overline{E}$ would be covered by $K_{Y_i}$-negative curves and would be contractible in $Y_i$. This is
impossible when $\rho(Y_i)=1$.

Denote by 
$f_i \colon Z_i=\Proj \bigoplus_{n\geq 0}
{h_i}_{\ast}\mathcal{O}_{X_i}(nE)\to Y_i$ a small partial
factorialization of $Y_i$; $Z_i$ is the symbolic blowup of $Y_i$
along the Weil non-Cartier divisor $\overline{E}$. Let $E'$ be the Cartier divisor $(h_i)^{-1}_{\ast}E$ on $Z_i$ and write $h_i=f_i\circ g_i$. Note that $Z_i$ has Picard rank $2$ and that $g_i$ contracts only curves of $X_i/Y_i$ that are disjoint from $E$.
\par
By the projection formula, $E'$ is  covered by $K_{Z_i}$-negative curves: there is an extremal contraction
$\psi_i \colon Z_i  \to Z_{i+1}$
whose exceptional divisor is $E'$. 
Consider the projective and surjective morphism
\[
\psi_i \circ g_i \colon X_i \to Z_{i+1},
\] 
and run a relative MMP on $X_i$ over $Z_{i+1}$.
The Contraction Theorem \cite{KMM87} shows that $\varphi_i$ factorizes 
$\psi_i\circ g_i$ and makes the diagram
\[
\xymatrix{E \subset X_i \ar[r]^{\varphi_i} \ar[d]_{g_i} & X_{i+1}\ar[d]^{g_{i+1}} \\
E'\subset Z_i \ar[d]_{f_i} \ar[r]^{\psi_i} & Z_{i+1} \\
Y_i &
}
\]
commutative. The morphism $g_{i+1}$ is crepant and $X_{i+1}$ and
$Z_{i+1}$ have the same anticanonical model. The Picard rank of
$Z_{i+1}$ is equal to $1$ and $A_{Z_{i+1}}$ is nef and big:
$Y_{i+1}=Z_{i+1}$.
\end{proof}
\begin{lem}
\label{lem:8}
Let $X$ be a weak-star Fano $3$-fold and $Y$ its
anticanonical model. If $\varphi \colon X \to \PS^1$ is an extremal del Pezzo
 fibration of degree $k$, then $k\geq 2$. If in addition, $R(X,A_X)$ is generated in degree $1$, then $k \geq 3$. 
\end{lem}
\begin{proof}
A general fibre $F$ of $\varphi$ is a nonsingular del Pezzo surface
of degree $k$. As $A_{X\vert F}= A_F$, the linear system $\vert A_X
\vert_{\vert F}$ is naturally a subsystem of $\vert A_F
\vert$. Since $\vert A_X \vert_{\vert F}$ is
basepoint free, the degree $k$ cannot be equal to $1$.  If $R(X,A_X)$ is generated in degree $1$, the morphism $\Phi_{\vert A_X \vert \vert F}$ is birational onto its image and $k \neq 2$. 
\end{proof}

If $\varphi_{n} \colon X_n \to S$ is a conic bundle, \cite{Cut88} shows
that $S$ is nonsingular.  
\begin{dfn}
 \label{dfn:4}
 A conic bundle $\varphi \colon X \to S$ is a \emph{weak Fano conic
  bundle} (resp.~ \emph{weak-star}) if $X$ is a weak
(resp.~ weak-star) Fano $3$-fold and $\rho(X/S)=1$.
\end{dfn}

I recall the following standard result.
\begin{lem}\cite{P05}
\label{lem:9}
If $\varphi \colon X\to S$ is a weak Fano conic bundle, $A_S$ is nef and big.
\end{lem}
\begin{proof}
Recall that the discriminant curve $\Delta$ of a conic bundle satisfies 
\[
4A_{S}= \varphi_{\ast}(A_X)^2 +\Delta,
\]
where $\varphi_{\ast} (A_X)^2$ is nef and big, and therefore $A_S$ is big.
Assume that $A_S$ is not nef and let $C \subset S$ be an irreducible
curve such
that $A_S \cdot C <0 $. The curve $C$ is necessarily contained in
$\Delta$, as $\varphi_{\ast}(A_X)^2$ is nef and $C^2\leq C\cdot \Delta <0$.
By adjunction:
\begin{align*}
-2 & \leq 2p_a(C)-2 \leq (A_S+C)\cdot C\\
  & \leq (A_S+\Delta)\cdot C \leq (3A_S-\varphi_{\ast}(A_X)^2)\cdot C \\
  & \leq3A_S\cdot C.
\end{align*}
This is impossible because $A_S \cdot C$ is an integer. 
\end{proof}
\begin{lem}
  \label{lem:10}
Let $\varphi \colon X \to S$ be a weak-star Fano conic bundle. Then
there is a weak-star Fano conic bundle $\varphi' \colon X' \to S'$,
with $S'= \PS^2, \F_0, \F_1$ or $\F_2$, such that the diagram \[
\xymatrix{ X \ar@{-->}[r] \ar[d]& X' \ar[d]\\
 S \ar[r]  & S' }
\]
is commutative.
\end{lem}
\begin{proof}
  If $S$ contains a $-1$-curve, denote its contraction by $S \to S'$. We may run a 
relative MMP of $X$ over $S'$ and this shows that there is
an extremal contraction $X \to X'$ fitting in a diagram as above. In
addition, $X' \to S'$ has a structure of weak-star conic bundle. 
There is a sequence of contractions of $-1$-curves $S \to S_1 \to
\cdots \to S_N$ where $S_N$ is either $\PS^2$
or a $\PS^1$-bundle over a curve $\Gamma$. 
The surface $S_N$ is clearly rational, so that if it is $\PS^1$-bundle over a curve $\Gamma$, $\Gamma$ is
isomorphic to $\PS^1$.
In that case, $S_N$ is a Hirzebruch
surface $\F_a= \PS(\mathcal{O}_{\PS^1}
\oplus\mathcal{O}_{\PS^1}(a))$. Since $A_{S_N}$ is nef, $S_N$ is either
$\F_0=\PS^1\times \PS^1$, $\F_1$ or $\F_2$. 
\end{proof}
\begin{thm}[End product of the MMP]
\label{thm:6}
Let $X=X_0$ be a weak-star Fano $3$-fold whose anticanonical
divisor $Y$ has Picard rank $1$.
The end product of the MMP on $X$ is one of:
\begin{itemize}
\item[1.] $X_n$ is a factorial terminal Fano $3$-fold with $\rho(X_n)=1$. 
\item[2.] $X_n \to \PS^1$ is an extremal del Pezzo fibration of degree $k$, with $2 \leq k$ and $\rho(X_n)=2$.
\item[3.] $X_n \to S$ is a conic bundle over $S=\PS^2, \PS^1
  \times \PS^1, \F_1$ or $\F_2$ and $\rho(X_n)=2$ or $3$.
\end{itemize}
\end{thm}
\subsection{A bound on the defect of some terminal Gorenstein Fano $3$-folds}

\begin{nt}
I say that a surface  $S \subset Y_{2g-2}$ is a plane (resp.~ a quadric) if its image by the anticanonical map is a plane (resp.~ a quadric) in $\PS^{g+1}$, that is when $(A_Y)^2\cdot S=1$ (resp.~ $2$). 
\end{nt}

\label{sec:bound-defect-some}
\begin{lem} 
\label{lem:5}
Let $X$ be a weak-star Fano $3$-fold whose
anticanonical model $Y$
does not contain an irreducible quadric surface. Then,
every divisorial extremal contraction $\varphi$ increases the
anticanonical degree
$(-K)^3=A^3$ by at least $4$.  
\end{lem}
\begin{proof}
Let $\varphi \colon X \to X'$ be a divisorial contraction. 
I use the notation of Mori-Cutkosky (see Theorem~\ref{Cut}) for the types of divisorial contractions.
The contraction $\varphi$ is of type E$1$ or E$2$. Indeed, $\varphi$ is not of type E$5$ because $X$ is weak-star Fano and it is not of type E$3$ or E$4$, because $Y$ does not contain an irreducible quadric surface.

If $\varphi$ is of type E$2$, $A_X=
\varphi^{\ast}(A_{X'})-2E$ and the degree increases by precisely $8$.

If $\varphi$ is of type E$1$, I claim that $A_{X'}^3 \geq A_X^3+4$.
Let $E$ be the exceptional divisor of $\varphi$ and $\Gamma$ its centre on $X'$.
As $Y$ does not contain an irreducible quadric surface, $(A_X)^2\cdot E\geq 3$, and the result follows from Equations $(1,2)$ in Lemma~\ref{lem:11}. 

\end{proof}

\begin{cor}
\label{cor:5}
Let $X$ be a weak-star Fano $3$-fold and $\varphi \colon X \to X'$ a
birational extremal contraction. Denote by $Y$ and $Y'$ the
anticanonical models of $X$ and $X'$. If $Y'$ contains an irreducible quadric surface $Q'$, then $Q'$ is disjoint from the centre of $\varphi$ and
$Y$ also contains an irreducible quadric surface $Q$.  
\end{cor}
\begin{proof}
This is a straightforward consequence of Step $4.$ of the proof of Theorem~\ref{thm:5}. 
\end{proof}
\begin{cor}
  \label{cor:1} 
Let $Y$ be a terminal Gorenstein Fano $3$-fold
of index $1$, Picard rank $1$ and of genus $g$.
If $Y$ does not contain a quadric or a plane, the defect of $Y$ is
bounded by $\big[\frac{12-g}{2}\big]+4$. 
\end{cor}
\begin{proof}
Let $h \colon X \to Y$ be a factorialization of $Y$; $X$ is a weak-star Fano $3$-fold. I prove that the Picard rank of $X$ is at most at most $\big[\frac{12-g}{2}\big]+5$.
\par
Run a MMP on $X$. I follow the notation of Lemma~\ref{lem:3}. The Picard rank of $X$ equals the number of divisorial contractions that occur during the MMP plus the Picard rank of the final $3$-fold $X_n$.  Lemma~\ref{lem:3} shows that at each intermediate step,  $Y_i$ has Picard rank $1$, and therefore $A_{X_i}^3=A_{Y_i}^3$ is equal to $2g-2$ with $2\leq g\leq 10$ or $g=12$ if $i(X_i)=1$, $8d$ with $1\leq d\leq 5$ if $i(X_i)=2$, $54$ if $i(X_i)=3$ and $64$ if $i(X_i)=4$. Note also that if $X_i \to X_{i+1}$ is a divisorial contraction, the index of $X_i$ is at most $2$ and if $i(X_i)=2$,  $i(X_{i+1})$ is divisible by $2$. 
By lemma~\ref{lem:5}, the number of divisorial contractions is bounded by $\big[\frac{12-g}{2}\big]+3$ if $A_{Y_n}^3\leq 40$ and by $\big[\frac{12-g}{2}\big]+4$ otherwise. The Picard rank of $X_n$ is at most $3$ and if $i(Y_n)=4$ (resp.~  if $i(Y_n)\geq 2$), $\rho(X_n)=1$ (resp.~ $\rho(Y_n)\leq 2$) \cite{Sh89}. The bound on the defect follows. 
\end{proof}
\begin{cor} 
\label{cor:2} 
 Let $Y$ be a terminal Gorenstein Fano $3$-fold of index $1$, Picard rank $1$ and of genus $g$. If $Y$ contains a
  quadric but does not contain a plane, the defect of $Y$ is bounded by $\big[\frac{12-n-g}{2}\big]+4+n$, where $n=\min\{\lfloor \frac{g+1}{3}\rfloor, 10-g\}$.
\end{cor}
\begin{proof}
Let $h \colon X \to Y$ be a factorialization of $Y$; $X$ is a
weak-star Fano $3$-fold. I prove that the Picard rank of $X$ is at
most $\big[\frac{12-n-g}{2}\big]+5+n$, where $n=\min\{\lfloor \frac{g+1}{3}\rfloor, 10-g\}$.
\par 
Run a MMP on $X$. If $Q \subset Y$ is an irreducible
reduced quadric lying on $Y$,  the index of $Y$ is $1$. Denote by $\widetilde{Q}$ its proper
transform on $X$. The quadric $\widetilde{Q}$ is negative on a
$K$-negative extremal ray $R$, the class of the proper transform of a
ruling of $Q$, and $R$ may be contracted. Let $\varphi \colon X \to X'$
be the contraction of $R$. If $\widetilde{Q}$ is contracted to a
point, the anticanonical degree increases by $2$. If $\widetilde{Q}$
is contracted to a curve $\Gamma$, by Lemma~\ref{lem:11}, $A_{X'}
\cdot \Gamma=2p_a(\Gamma)$ and
$A_X^3=A_{X'}^3-2(p_a(\Gamma)+1)$. The contraction $\varphi$ increases
the anticanonical degree by at least $2$. 
\par 
By Lemma~\ref{lem:5}, if $Y_i$ contains a quadric $Q_i$, $Q_i$ does not intersect the centre of any previous extremal contraction. Denote by  $\widetilde{Q_i}$ the proper transform of $Q_i$ on $X$, $h(\widetilde{Q_i})$ is a quadric. We may therefore assume that quadrics are contracted first. While quadrics are contracted, the index of $X_i$ is $1$, and as in the proof of Corollary~\ref{cor:1}, the anticanonical degrees $A_{X_i}^3$ can only be of the form $2g-2$, with $2\leq g\leq 10$ or $g=12$. There are hence at most $10-g$ divisorial contractions that decrease the anticanonical degree by $2$. Furthermore, the quadrics $h(\widetilde{Q_i})\subset Y \subset \PS^{g+1}$ are disjoint, so that at most $\lfloor \frac{g+1}{3}\rfloor$ quadrics may be contracted. If $X_i\to X_{i+1}$ is the contraction of a quadric, $i(X_i)=i(X_{i+1})=1$ and $A_{Y_{i+1}}^3=2g-2$ with $2\leq g\leq 10$ or $g=12$. The result then follows from Corollary~\ref{cor:1}.
 \end{proof}
\subsection{Higher index Fano $3$-folds}
\label{sec:higher-index-fanos}
Shin classifies canonical Gorenstein Fano $3$-folds of index
greater than $1$\cite{Sh89}. If $Y$ has terminal Gorenstein
singularities and index $2,3$ or $4$, it has Picard rank $1$ and
if $X \to Y$ is a small factorialization, $X$ is a weak-star Fano
$3$-fold. 

\begin{lem}
\label{lem:2}
Let $X$ be an index $2$ weak-star Fano $3$-fold and let $\varphi \colon
X \to X'$ be the contraction of an extremal ray. Then, one of the
following holds:
\begin{itemize}
\item[1.] $\varphi$ is birational, and $\varphi$ is either a flop or an E$2$
contraction. 
\item[2.] $\varphi\colon X\to S$ is an \'etale $\PS^1$-bundle.
\item[3.] $\varphi\colon X \to \PS^1$ is a del Pezzo fibration of degree $8$,
  and $\varphi$ is a quadric bundle.  
\end{itemize}
\end{lem}
\begin{rem}
\label{rem:3}
 Similarly, if $X$ has index $3$ and $\rho(X)\geq 2$, then $\rho(X)=2$.
 Any contraction of an extremal ray is either a flop or a del
 Pezzo fibration of degree $9$, that is, $X$ is a $\PS^2$-bundle over
 $\PS^1$.  
\end{rem}   
\begin{cor}
\label{cor:3}
 Let $Y$ be a terminal Gorenstein Fano $3$-fold of
 index $2$ and $X$ a small factorialization of
$Y$. Denote by $h^3= \frac{A_X^3}{8}$ the degree of $X$ and $Y$. 
The Picard rank of $X$ is at most $8-h^3$. 
The $3$-fold $X$ contains at most $7-h^3$ disjoint planes.
\end{cor}
\begin{proof}
Each divisorial contraction increases the degree by $1$: this proves
the first assertion.
\par 
Let $E= \PS^2$ be a plane contained in $X$, then $A_{X \vert \PS^2}=
 \mathcal{O}_{\PS^2}(2)$. 
The contraction theorem shows that there is an extremal contraction that
contracts $E$ to a point. 
 \par 
For a suitable choice of $X$, all birational contractions appearing in the MMP are divisorial.
Indeed, a flopping
contraction and a divisorial contraction of type
E$2$ always commute \cite{CJR08}, therefore we may assume that flops are performed first.
The proper transforms on $X$ of the 
divisors contracted when running the MMP on $X$ are
disjoint planes. There is no other plane disjoint from these lying on $X$,
as this would give rise to another extremal contraction.
\end{proof}
\section{Index $1$ Fano $3$-folds that contain a plane}
\label{sec:fano-3-folds-4}
\par 
This section is independent of the rest of the paper.  
I consider terminal Gorenstein Fano $3$-folds whose small factorializations are not weak-star Fanos: the rank of their divisor class group cannot be bounded with the techniques developed in the preceding sections. I first bound the rank of the divisor class group of quartic hypersurfaces by explicit calculation and then indicate how to treat other Fano $3$-folds with Picard rank $1$.
 
\subsection{$Y_4 \subset \PS^4$, genus $3$}
\label{sec:y_4-subset-ps4}
Let $Y=Y_4^3 \subset \PS^4$ be a quartic $3$-fold with terminal
singularities that contains a plane $\Pi{=} \{x_0{=}x_1{=}0\}$. As $Y$ is a
hypersurface in $\PS^4$, it is Gorenstein and has Picard rank $1$.
The equation of $Y$ is of the form:
\begin{equation}
\label{eq:3}
Y{=} \{x_0 a_3(x_0, \cdots , x_4)+
x_1 b_3(x_0,\cdots, x_4)=0 \}.  
\end{equation}
\par 
Let $X$ be the blowup of $Y$ along the plane $\Pi$. Locally, the
equation of $X$ can be written:
\begin{multline}
  \label{eq:4}
\{ t_0 a_3(t_0x, t_1x, x_2, x_3, x_4)+ t_1 b_3(t_0x, t_1x
, x_2, x_3, x_4)=0 
\} \\\subset \PS (t_0, t_1) \times \PS (x, x_2, x_3, x_4),  
\end{multline}
where the variable $x$ is defined by $x_0=t_0x$, $x_1=t_1x$. 
\par
The $3$-fold $X$ has a natural structure of cubic del Pezzo fibration
over $\PS^1$. The generic fibre $X_{\eta}$ is reduced and
irreducible. 
However,
special fibres might be reducible. The cubic fibration $X$ is a small
partial factorialization of $Y$, in particular the ranks of $\Cl X$ and of $\Cl Y$ are equal.
\begin{lem}
\label{lem:13}
Denote by $N$ the number of reducible fibres with $3$ irreducible
components and by $M$ the number of reducible fibres with $2$ irreducible
components. The rank of $\Cl X$ is bounded by $8+2N + M$.
\end{lem}
\begin{proof}
 Let $\mathcal{O}$ be the local ring of $\PS^1$ at a point $P$
corresponding to a reducible fibre of $X$. Let $K$ be its fraction
 field, $t$ its parameter and $k=\C$ its residue field. Let 
$S = \Spec(\mathcal{O})$, $\eta= \Spec K$ 
the generic
point and let $0=\Spec k$ be the origin.  
Let $\PS= \PS^n_{\mathcal{O}}$ be an $n$-dimensional projective
space over $S$ and $L=L_d$ a $d$-dimensional projective subspace,
defined over $k$. If $d \leq n-1 $, there is a birational
transformation of $\PS$ centred at $L$:
\[\xymatrix{
(x_0{:}\ldots{:}x_n) \ar@{|->}[r] &(tx_0{:} \ldots {:}tx_d{:}x_{d+1}{:} \ldots {:}x_n) };
\]
$\phi_L$ is the projection from $L$.  
\par 
Corti shows \cite[Flowchart 2.13, Lemma 2.17 and Corollary
2.20]{Co96} that, if $X_{\mathcal{O}} \subset \PS_{\mathcal{O}}^3$ is
a weak Fano cubic fibration with small anticanonical map, projecting
away from planes contained in
the central fibre yields a standard integral model $X'_{\mathcal{O}}$ for
$X_{\eta}$, i.e.~ a flat subscheme
$ X'_{\mathcal{O}} \subset \PS_{\mathcal{O}}^3$ with isolated cDV singularities and
reduced and irreducible central fibre.  
Each projection from a plane
contained in the central fibre only affects the cubic fibration in the
central fibre and it strictly decreases the number of $k$-irreducible
components of the central fibre.
\par
If a fibre of $X \to \PS^1$ is reducible, it has at most $3$ irreducible
components (one of which at least is a projective plane). Project away from
planes in reducible fibres of $X$ in order to obtain a standard integral
model. The lemma then follows from the sequence:
\[
0 \to \Z[\pi^{\ast} \mathcal{O}_{\PS^1}(1)] \to \W(X')\to \Pic
(X'_{\eta})\to 0,
\]
which is exact when $\pi \colon X' \to \PS^1$ has reduced and irreducible fibres:
\end{proof}

Bounding the rank of $\Cl Y$ is therefore equivalent to bounding the number of reducible fibres of $X$.
\begin{thm}
  \label{thm:14}
Let $Y {=}Y_4^3\subset \PS^4$ be a terminal quartic $3$-fold that contains a plane
$\Pi$ and let $X$ be the cubic del Pezzo fibration obtained by
blowing up $Y$ along $\Pi$. The cubic fibration $X$ has at most $4$
reducible fibres : the rank of $\Cl Y$ is at most $16$.  
\end{thm}
\begin{proof}  
Let 
\[
H_{(\lambda{:}\mu)}= \{ \lambda x_0+\mu x_1{=}0\} \subset \PS^4, (\lambda{:}\mu)\in \PS^1
\]
 be the pencil of hyperplanes of $\PS^4$ that contain $\Pi$.
 
The hyperplane section of $Y\subset \PS^4$ corresponding to
$H_{(\lambda{:} \mu)}$ is a reducible quartic surface that contains the
plane $\Pi$; more precisely $Y \cap H_{(\lambda{:}\mu)} = \Pi \cup
Y'_{(\lambda{:} \mu)}$, where $Y'_{(\lambda{:} \mu)}$ is a possibly
reducible cubic surface, naturally isomorphic to the fibre
$X_{(\lambda{:} \mu)}$.

If $X$ has a reducible fibre over $(\lambda{:}\mu)\in \PS^1$, 
$X$ contains a plane either of the form: 
\[
\Pi_{(\lambda{:}\mu)}=\{\lambda t_0 +\mu t_1= l(x_2, x_3, x_4)+l'(t_0x, t_1x)=0 \},
\]
where $l$ and $l'$ are linear, or of the form:
\[
\Pi_{(\lambda{:}\mu)}=\{ \lambda t_0 +\mu t_1=x=0 \}.
\]
The plane $\Pi_{(\lambda{:} \mu)}\subset X_{(\lambda{:}\mu)}$ is
corresponds to a plane 
$\Pi'_{(\lambda{:}\mu)}$ lying on $Y\cap H_{(\lambda{:}\mu)}$. 
In the first case,
  $\Pi'_{(\lambda{:}\mu)}$ is of the form:
\[
\Pi'_{(\lambda{:}\mu)}=\{\lambda x_0 +\mu x_1= l(x_2, x_3, x_4)+l'(x_0, x_1)=0 \}
\] 
and meets $\Pi$ in a line $L_{(\lambda{:}\mu)}=\{x_0=x_1= l(x_2, x_3,
x_4)=0\}$, while in the second case $\Pi'_{(\lambda{:}\mu)}= \Pi$ and $H_{(\lambda{:}\mu)}$ is tangent to $Y$ along $\Pi$.

\noindent
\emph{Step 1.}
I first show that no hyperplane is tangent to $Y$ along $\Pi$.

If this is the case, then without loss of generality, we may assume that the hyperplane $H_{(1{:}0)} $ is tangent to $Y$ along $\Pi$. The equation of $Y$ then reads
\[
\{x_0 a_3+ x_1x_0 a_2+x_1^2 b_2=0\}
\]
where $a_3$ is a homogeneous form of degree $3$ and $a_2$ and $b_2$ are homogeneous forms of degree $2$. The quartic $Y$ is singular along the curve $\Gamma{=}\{a_3{=}x_0{=}x_1{=}0\}$: this contradicts $Y$ having terminal singularities.

\noindent
\emph{Step 2.}
I show that if $\Pi'_1$ and $\Pi'_2$ are two planes distinct from $\Pi$ that lie on distinct hyperplane sections of $Y$, then they intersect $\Pi$ in distinct lines. 

Let $L_1=\Pi\cap \Pi'_1$ and $L_2=\Pi\cap \Pi'_2$.
If $L_1=L_2$, up to coordinate change on $\PS^4$, I may assume
that $\Pi'_1{=}\{x_0{=} x_2{=}0\}$ and $\Pi'_2{=}\{x_1{=}x_2{=}0\}$.
The plane $\Pi'_1$ (resp.~ $\Pi'_2$) lies on $Y$ if and only if, in
\eqref{eq:3}, the
homogeneous form $b_3$ is in the ideal $\langle x_0,x_2 \rangle$ of $\Pi'_1$
(resp.~ the homogeneous form $a_3$ is in the ideal 
$\langle x_1, x_2 \rangle$ of
$\Pi'_2$). The quartic $Y$ then has multiplicity $2$ along the line
$L{=} \{x_0{=}x_1{=}x_2{=}0\}$. This is a contradiction: $\Pi'_1$ and $\Pi'_2$
meet at a point.

\noindent
\emph{Step 3.}
I now show that if $\Pi'_1, \Pi'_2$ and
$\Pi'_3$ are $3$ planes that lie on distinct hyperplane sections of $Y$ and are distinct from $\Pi$, the lines $L_1= \Pi\cap\Pi'_1, L_2= \Pi\cap \Pi'_2$ and $L_3=\Pi\cap \Pi'_3$ are not concurrent.

Assume that the planes $\Pi'_1, \Pi'_2$ and
$\Pi'_3$ meet at a point $P$.
Up to coordinate change on $\PS^4$, we may assume that:
\begin{eqnarray*}
\Pi'_1=&\{x_0{=}x_2{=}0\}\\ 
\Pi'_2=& \{x_1{=}x_3{=}0\}\\
\Pi'_3=& \{x_0+x_1{=} x_2+x_3+l(x_0,x_1){=}0\}, 
\end{eqnarray*}
where $l$ is a linear form.
The equation of $Y$ is in the ideal spanned by the monomials:
\begin{eqnarray*}
I =&\langle (x_0+x_1)x_0x_1, x_0x_1(x_2+x_3+l(x_0,x_1)), (x_0+x_1)x_1x_2,
(x_0+x_1)x_0x_3, \\ & x_1x_2(x_2+x_3+l(x_0,x_1)),
 x_0x_3(x_2+x_3+l(x_0,x_1)), (x_0+x_1)x_2x_3,\\
 &x_2x_3(x_2+x_3+l(x_0,x_1))
\rangle, 
\end{eqnarray*}
hence $P=(0{:}0{:}0{:}0{:}1)= L_1\cap L_2\cap L_3 \in Y$ is not a cDV point. 
If $X$ contains at least three distinct reducible fibres,  then up to
coordinate change on $\PS^4$, we may assume that
\begin{eqnarray*}
\Pi'_1=&\{x_0{=} x_2{=}0 \} \subset H_{(0{:}1)}\\
\Pi'_2=&\{x_1{=} x_3{=}0 \}\subset H_{(1{:}0)}\\
\Pi'_3=&\{x_0+x_1{=} x_4{=}0 \} \subset H_{(1{:}1)}.  
\end{eqnarray*}
lie on $Y$.

\noindent
\emph{Step 4.}   
Finally, I prove that $X$ has at most $4$ reducible fibres.

Assume that $X$ has at least $5$ reducible fibres and let $\Pi'_1, \cdots, \Pi'_5 \subset
\PS^4$ be planes that are distinct from $\Pi$ and that lie on distinct hyperplane sections of $Y$.
 Denote by $L_i$ the line $\Pi\cap
\Pi'_i$ for $1\leq i\leq 5$. Steps $2$ and $3$ show that the lines
$L_i \subset \Pi$ are distinct and that no three of these lines are concurrent. 

We may therefore assume that:
\begin{eqnarray*}
\Pi'_1=&\{x_0{=} x_2{=}0 \}\\
\Pi'_2=&\{x_1{=} x_3{=}0 \}\\
\Pi'_3=&\{x_0+x_1{=} x_4{=}0 \} \\
\Pi'_4=&\{x_0+\lambda x_1{=} ax_2+bx_3+cx_4+ l(x_0, x_1){=}0 \}
\end{eqnarray*}
where $a, b$ and $c$ are constants and $l$ is a linear form. The
constant $\lambda \neq 0,1$ or $\infty$ because no two of the planes
$\Pi_i$ lie in the same hyperplane section of $Y$. The
constants $a, b$ and $c$ are all non-zero,  as otherwise either two of
the lines $L_i$ coincide, or three of the lines $L_i$ meet at a
point. Up to rescaling, we may assume that:
\begin{equation*}
\Pi'_4{=}\{t_0+\lambda t_1{=} x_2+x_3+x_4+ l(t_0, t_1)x{=}0 \}.
\end{equation*}
The equation of $\Pi'_5$ is of
the form:
\[
\Pi'_5{=}\{x_0+\mu x_1{=} \alpha x_2+ \beta x_3 + \gamma x_4+
l'(x_0,x_1){=}0 \},
\] 
where $l'$ is a linear form and $\alpha, \beta$ and $\gamma$ are
constants. As any three lines of $L_1,\cdots , L_5$ have to satisfy
the conditions of Steps $2$ and $3$, $\alpha, \beta$ and $\gamma$ are
all non-zero. I may assume that $\alpha=1$. Considering triples of
lines $(L_4, L_5, L_i)$ for $1\leq i\leq 3 $ shows that $\beta\neq 1$,
$\gamma\neq 1$, and $(\beta{:}\gamma)\neq (1{:}1)$.
The equation of $Y$ may be written uniquely in the form:
\begin{equation}
\label{eq:6}
 f(x_0,\cdots, x_4) + g(x_0, \cdots, x_4)=0,  
\end{equation}
where the monomials that appear in $f$ have degree at least $2$ in $x_0, x_1$, while $g(x_0, \cdots, x_4)=x_0g_0(x_2, x_3,x_4)+x_1g_0(x_2, x_3,x_4)$ has degree exactly $1$ in $x_0, x_1$.
In the expression \eqref{eq:6}, $g$
may be written as a linear combination of the monomials:
\begin{multline*}
\langle x_0x_3x_4(x_2+x_3+x_4),
x_1x_2x_4(x_2+x_3+x_4),\\ (x_0+x_1)x_2x_3(x_2+x_3+x_4), (x_0+\lambda x_1)
x_2x_3x_4\rangle.
\end{multline*}
 or as a linear combination of the monomials: 
\begin{multline*}
\langle x_0x_3x_4(x_2+\beta x_3+\gamma x_4),
x_1x_2x_4(x_2+\beta x_3+\gamma x_4),\\(x_0+x_1)x_2x_3(x_2+\beta
x_3+\gamma x_4),(x_0+\mu x_1)
x_2x_3x_4 \rangle.  
\end{multline*}
\par 
Equating these two expressions shows that either $g=0$ or
$\lambda=\mu$. If $g=0$, $Y$ does not have isolated singularities; If $\lambda= \mu$, $\Pi'_4$ and $\Pi'_5$ are bot contained in the
hyperplane section $H_{(1{:}\lambda)}$. Both cases yield a contradiction: the cubic fibration $X$ has at most $4$ reducible fibres. 
\end{proof}
\begin{exa}[The Burkhardt quartic]
\label{exa:1}  
It is known that a nodal quartic hypersurface in $\PS^4$ has at most $45$ nodes
\cite{V83,F86}. Up to projective equivalence, there
is only one $45$-nodal quartic $3$-fold $Y\subset \PS^4$
\cite{dJSBVV90} and its equation is:
\[ 
 \{x_0^4 -x_0 (x_1^3+x_2^3+x_3^3+x_4^3)+ 3x_1x_2x_3x_4 = 0 \}. 
\]
Let $\widetilde{Y}$ be the blowup of $Y$ at the $45$ nodes.
The defect of $Y$ satisfies $\sigma(Y)=b_2(\widetilde{Y})-45-b_2(Y)=
b_2(\widetilde{Y})-46$.
The cohomology of $\widetilde{Y}$ is
determined in \cite{HM01} and $b_2(\widetilde{Y})=61$, so that
the defect of the Burkhardt
quartic is $15$.
\par 
Alternatively, the plane $\Pi{=}\{x_0{=}x_1{=}0\}$ is contained in $Y$. Let
$X$ be the $3$-fold obtained by blowing up $Y$ along the plane $\Pi$;
the equation of $X$ is:
\[
t_0(t_0^3-t_1^3)x^3-t_0(x_2^3+x_3^3+x_4^3)+ 3t_1x_2x_3x_4=0,
\] 
and on the affine piece $t_1 \neq 0$, this reads
\[
t_0(t_0^3-1)x^3-t_0(x_2^3+x_3^3+x_4^3)+ 3x_2x_3x_4=0.
\] 
The central fibre, which corresponds to  $(t_0{:}t_1)=(0{:}1)$, has three
irreducible components: the planes $\{ t_0{=}x_i{=}0\}$.
Each fibre is reduced and irreducible for $(t_0{:}t_1) \neq (0{:}1)$ and
$(t_0{:}t_1) \neq (1{:}\omega^i)$ for $0 \leq i \leq 2$, where $\omega$ is
a cube root
of unity (notice that considering the affine piece $t \neq 1$ yields
the same results). The generic fibre $X_{\eta}$ is a nonsingular
cubic surface in $\PS^3$. 
\par 
Consider for instance the fibre $X_1$ over $(1{:}1)$. The fibre $X_1$ is the
union of three planes in $\PS(x, x_2, x_3, x_4)$: $\Pi_1{=}\{
x_2+x_3+x_4{=}0\}$, $\Pi_2{=}\{\omega x_2+x_3+\omega^2 x_4{=}0\}$ and
$\Pi_3{=}\{\omega^2 x_2+x_3+\omega x_4{=}0\}$. The situation is
analogous for the other two reducible fibres.
There are $27$ closed subschemes in
$X \smallsetminus\bigcup_{i=0,1,2}
X_i$ isomorphic to $\PS^1_{\PS^1 \smallsetminus
  (\cup_i\{(\omega^i{:}1)\})}$. In other words, the $27$
lines on the generic fibre may be completed to divisors on $X$, they
are rational over $\PS^1 \smallsetminus (\cup_i\{(\omega^i{:}1)\})$. The
Picard rank of $X_{\eta}$ is $7$ and the generators of
$\Pic(X_{\eta})$ complete to independent Cartier divisors on $X$. 
\par 
The rank of the Weil group of the  Burkhardt quartic is $16$ and its
defect is $15$. 
\end{exa}
\begin{cor}
Let $Y_4\subset \PS^4$ be a quartic $3$-fold with defect $15$. Then $Y$ is projectively equivalent to the Burkhardt quartic.
\end{cor}
\begin{proof}
The quartic $Y_4\subset \PS^4$ has defect $15$, hence by Corollary~\ref{cor:2}, $Y$ contains a plane $\Pi{=}\{x_0{=}x_1{=}0\}$. Denote by $X$ the blowup of $Y$ along the plane $\Pi$. By Lemma~\ref{lem:13}, the cubic fibration $X$ has exactly $4$ reducible fibres, each of which has exactly three irreducible components. As in the proof of Theorem~\ref{thm:14}, we use the description of planes contained in reducible fibres of $X$ to carry a refined analysis of the monomials appearing in the equation of $X$ and $Y$. If the equation of $Y$ reads
\[
\{x_0a_3+x_1b_3=0\},
\]
where $a_3, b_3$ are homogeneous cubic polynomials, and if $X$ has $4$ reducible fibres,  it is easy to see that $\{x_0=x_1=a_3=b_3=0\}$ consists of exactly $9$ points, which hence have to be ordinary double points. There are $9$ flopping lines in $X$ that lie above these $9$ ordinary double points. The cubic fibration $X$ is nonsingular outside of its $4$ reducible fibres. Using calculations analogous to those in the proof of Theorem~\ref{thm:14}, there are precisely $9$ singular points lying on each reducible fibre, and these therefore are ordinary double points. The quartic $Y$ is nodal and has at least $45$ ordinary double points. By \cite{V83,dJSBVV90}, $Y$ is projectively equivalent to the Burkhardt quartic. 

\end{proof}
\subsection{Higher genera}
The genus $g$ of $Y$ satisfies $3 \leq g \leq 10$ or $g=12$
\cite{Mu02}.
Mukai shows that if $g=12$, $Y$ is nonsingular; $Y$ is factorial and
therefore does not contain a plane.
If $g=10$, $Y_{18} \subset \PS^{11}$ is a linear section of a
$5$-dimensional $G_2$-variety in $\PS^{13}$; . If $g=9$, $Y_{16} \subset
\PS^{10}$ is a linear section of the $6$-dimensional symplectic
grassmannian $G \subset \PS^{12}$. In either case, $Y$ does not contain
a plane, because neither $G_2\subset \PS^{13}$ nor $G \subset \PS^{12}$ does \cite{LaMa03}.

Let $Y_{2g-2}\subset \PS^{g+1}$ for $4\leq g\leq 8$ be a Fano $3$-fold with terminal Gorenstein singularities that contains a plane $\Pi\simeq \PS^2$. The defect of $Y$ can be bounded by studying the projection from the plane $\Pi\subset Y$. Let $\mu \colon X\to Y$ be the blowup of $Y$ along $\Pi$, so that the projection from $\Pi$ decomposes as:
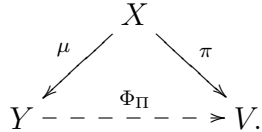
\[
\xymatrix{\quad & X \ar[dl]_{\mu} \ar[dr]^{\pi} & \quad \\
Y \ar@{-->}[rr]^{\Phi_{\Pi}} & \quad & V.}
\]
\begin{enumerate}
\item[1.] $Y_{2,3}\subset \PS^5, g=4$. The morphism $\pi$ is $K_X$-negative and is a conic bundle over $V=\PS^2$. The rank of the Weil group of $X$ is bounded above by $1+\deg(\Delta)$, where $\Delta$ is the discriminant curve of $\pi$ in $\PS^2$. The $3$-fold $X$ is a complete intersection of sections of the linear systems $\vert M+L\vert$ and $\vert 2M+L\vert$ on the scroll $\F_{0,0,1}$ over $\PS^2$. Computations on the scroll show that $\deg(\Delta)\leq 7$, so that the defect of $Y$ is bounded by $8$.
\item[2.] $Y_{2,2,2}\subset \PS^6, g=5$. The morphism $\pi$ is birational, maps onto $V=\PS^3$ and $A_X$ is $\pi$-ample. The $3$-fold $X$ is a complete intersection of members of $\vert M+L\vert$ on the scroll $\F_{0,0,1}$ over $\PS^3$.  In order to bound the defect of $Y$, it is enough to bound the numbers of points of $\PS^3$ such that the three sections of $\vert M+L\vert$ in the fibre of the scroll $\F_{0,0,1}$ are not linearly independent. Indeed, the centre of the morphism $\pi$ consists of a finite number of points because $Y$ has isolated singularities. Crude considerations on the birational map $X \to \PS^3$ show that the defect of $Y$ is bounded above by $7$.
\item[3.] $g=6,7,8$. The morphism $\pi$ is birational, $A_X$ is $\pi$-ample, and maps $X$ onto $V_{d}$, a Fano $3$-fold of index $2$ and degree $d=3$ when $g=6$, $4$ when $g=7$, and $5$ when $g=8$. The $3$-fold $V_d$ has isolated canonical singularities. We can adapt some of the methods in \cite{T06}, and use the explicit descriptions of linear subspaces of $Y$ that follow from \cite{Mu02} to study the birational map $X \to Y$. The defect of $Y$ is bounded above by $5$ when $g=6, 7$ or $g=8$.   
\end{enumerate}

\begin{pro}
\label{pro:7}\mbox{}
\begin{enumerate}
  \item[1.]
 The defect of a terminal Gorenstein
 $Y_{2,3}\subset \PS^5$ that contains a plane is at most $8$,  
  \item[2.]
 The defect of a terminal Gorenstein
 $Y_{2,2,2}\subset \PS^6$ that contains a plane is at most $7$,
\item[3.]
 The defect of a Picard rank $1$ terminal Gorenstein
 Fano $3$-fold of genus $g =6,7$ or $8$ that contains a plane is at most $5$.
\end{enumerate}
\end{pro}

\subsection{Non anticanonically embedded Fano $3$-folds}
I say a few words about the Picard rank $1$ terminal Gorenstein Fano $3$-folds whose anticanonical ring are not generated in degree $1$ and that contain a plane. First recall from the proof of Lemma~\ref{lem:4} that the case of the double quadric does not occur, so that I only need to consider the case $g=2$.  
It is known that the defect of a double sextic with no worse than ordinary double points is at most $13$ and that this bound is attained by a double cover of $\PS^3$ ramified along the Barth sextic surface \cite{W98,E99}. I conjecture that this bound holds for terminal Gorenstein singularities, although I have little evidence to support this. 
\begin{con}
\label{con:1}
  The defect of a double sextic with terminal Gorenstein singularities is at most $13$.
\end{con}

\bibliography{bib}
\end{document}